\documentclass[a4paper,final]{article}

\usepackage{url}
\urldef{\miaemail}\url{gennaro.amendola@uniecampus.it}

\usepackage{booktabs}

\usepackage{amsthm}
\newtheorem{theorem}{Theorem}
\newtheorem{lemma}[theorem]{Lemma}

\theoremstyle{definition}
\newtheorem{defi}[theorem]{Definition}
\newtheorem{remark}[theorem]{Remark}

\usepackage{amssymb}
\usepackage{amsmath}
\usepackage{mathtools}

\usepackage{graphicx}

\usepackage{url}

\usepackage{tikz}

\usetikzlibrary{3d}

\tikzset{coordinate 3d 30gradi/.style={x={(210:0.544cm)}, y={(1cm,0cm)}, z={(0cm,1cm)}, scale=2} }

\usetikzlibrary{arrows.meta}

\tikzset{<-> /.tip = Stealth[] }

\tikzset{pics/punto/.style={ code={ \fill (0,0) circle [radius=1.5pt]; } } }

\tikzset{figura cubo/.style={scale=0.8} }

\tikzset{
lato/.style={draw, thin},
lato tratteggiato/.style={draw, thin, dashed},
diagonale/.style={draw, thick},
diagonale tratteggiato/.style={draw, thick, dashed},
lato evidenziato/.style={draw, thick},
lato evidenziato tratteggiato/.style={draw, thick, dashed},
freccia incollamento/.style={draw, thin, <->},
}

\newcommand{\letteraincollamento}[1]{\Large{#1}}

\newcommand{\matRP}{\mathbb{RP}}

\newcommand{\ptwoirred}{$\mathbb{P}^2$-irreducible}
\newcommand{\Lthreeone}{L_{3,1}}
\newcommand{\Lfourone}{L_{4,1}}

\newcommand{\co}{\colon\thinspace}

\title{Non-orientable $3$-manifolds\\ of surface-complexity one}
\author{Gennaro {\sc Amendola}\thanks{Department of Theoretical and Applied Sciences, eCampus University, Novedrate, Italy\newline
Member of GNSAGA of INDAM\newline
\miaemail}}
\date{}

\begin{document}
\maketitle

\begin{abstract}
We classify all closed non-orientable \ptwoirred\ 3-manifolds obtained by identifying the faces of a cube.
These turn out to be the closed non-orientable \ptwoirred\ 3-manifolds with surface-complexity one.
We show that they are the four flat ones.
\end{abstract}

\begin{center}
\end{center}

\begin{center}
\end{center}

\section*{Introduction}

The study of closed 3-manifolds constructed by identifying the faces of a cube started with Poincaré~\cite{Poincare} in 1895 to produce examples of manifolds in the study of the fundamental group and of the Betti numbers.
In this paper we will deal with the non-orientable case by starting a classification process.

Non-orientable 3-manifolds seem to be much more sporadic than orientable ones.
For instance, among the 8 three-dimensional geometries, only 5 have non-orientable representatives~\cite{Scott}.
Moreover, among cusped hyperbolic 3-manifolds of Matveev complexity up to nine, only $14045$ of $75956$ are non-orientable, as shown in the Callahan-Hildebrand-Weeks-Thistlethwaite-Burton census~\cite{CaHiWe,Thistlethwaite,Burton:hyperbolic_census_9}.
Also, among closed \ptwoirred\ 3-manifolds of Matveev complexity up to seven, only $8$ of $318$ are non-orientable~\cite{Amendola-Martelli:non_orient_7,Burton:small_non_orient}.
Eventually, all three closed \ptwoirred\ 3-manifolds of surface-complexity zero are orientable~\cite{Amendola:surf_compl}.
Here we show that, among closed \ptwoirred\ 3-manifolds of surface-complexity (and hence cubic-complexity) one, only $4$ of $15$ are non-orientable.

We refer to three different complexities on 3-manifolds, used to carry out the classification processes.
The Matveev complexity was defined in~\cite{Matveev:compl}, and in the cases described above equals the minimum number of tetrahedra needed to triangulate the manifold if it is distinct from the sphere $S^3$, the projective space $\matRP^3$ and the Lens space $\Lthreeone$ (having Matveev complexity zero)~\cite{Martelli-Petronio:decomposition}.
The cubic-complexity is the minimum number of cubes needed to cubulate the manifold (i.e.~to construct the manifold by gluing cubes along the boundary squares)~\cite{Tarkaev:cubic_compl}.
The surface-complexity is the minimum number of triple points needed by the image of a transverse immersion of a closed surface to divide the manifold into balls, and is equal to the cubic-complexity under some hypotheses on the manifold, but it is more flexible~\cite{Amendola:surf_compl,Amendola:surf_compl_cpt}.
For the sake of completeness, we recall that analogous interesting definitions involving surface immersions are the Montesinos complexity and the triple point spectrum, given by Vigara~\cite{Vigara:calculus} and studied by Lozano and Vigara~\cite{Lozano-Vigara:subadditivity,Lozano-Vigara:representing,Lozano-Vigara:triple_point_spectrum}.

In this paper, we classify all closed non-orientable \ptwoirred\ 3-manifolds with surface-complexity (and hence cubic-complexity) one: in some way they are the ``simplest'' ones, because is turns out that they are the four flat ones.

Usually the classification process is computer-aided.
In this case one could study all $8^3=512$ possible gluings for the boundary squares of the cube and identify the object obtained (which may not be a manifold), but we have preferred to use some simple theoretical results to simplify the search among the manifolds with Matveev complexity up to six, avoiding hence the complete enumeration.
We plan for a subsequent paper to continue the classification process with the aid of a computer.

\section{Definitions}

Throughout this paper, all 3-manifolds are assumed to be connected and closed.
By $M$, we will always denote such a (connected and closed) 3-manifold.
Using the {\em Hauptvermutung}, we will freely intermingle the differentiable, piecewise linear and topological viewpoints.

A {\em cubulation} of $M$ is a cell-decomposition of $M$ such that
\begin{itemize}
\item each 2-cell (called a {\em square}) is glued along 4 edges,
\item each 3-cell (called a {\em cube}) is glued along 6 faces arranged like the boundary of a cube.
\end{itemize}
Note that self-adjacencies and multiple adjacencies are allowed.
In the figures we have used (non-symmetric) letters to show the gluing information.
In Fig.~\ref{fig:cubul_example} we have shown a cubulation of the 3-dimensional torus $S^1\times S^1\times S^1$ with one cube.
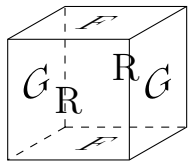
\begin{figure}[t]
  \centerline{
  \begin{tikzpicture}[coordinate 3d 30gradi, figura cubo]
\draw [lato] (1,0,0) -- (1,1,0) -- (0,1,0) -- (0,1,1) -- (0,0,1) -- (1,0,1) -- cycle;
\draw [lato tratteggiato] (0,0,0) -- (0,1,0) (0,0,0) -- (1,0,0) (0,0,0) -- (0,0,1);
\draw [lato] (1,0,1) -- (1,1,1) (1,1,1) -- (0,1,1) (1,1,0) -- (1,1,1);
\draw (0,0.5,0.5) node {\letteraincollamento{R}};
\draw (1,0.5,0.5) node {\letteraincollamento{R}};
\draw (0.5,0,0.5) node[yslant=0.58] {\letteraincollamento{G}};
\draw (0.5,1,0.5) node[yslant=0.58] {\letteraincollamento{G}};
\draw (0.5,0.5,0) node[xslant=1.73, yscale=0.5] {\letteraincollamento{F}};
\draw (0.5,0.5,1) node[xslant=1.73, yscale=0.5] {\letteraincollamento{F}};
\end{tikzpicture}}
  \caption{A cubulation of the 3-dimensional torus $S^1\times S^1\times S^1$ with one cube (the letters show that the identification of each pair of opposite faces is the obvious one, i.e.~the one without twists and reflections).}
  \label{fig:cubul_example}
\end{figure}
\begin{defi}
The {\em cubic-complexity} of $M$ is equal to $c$ if $M$ possesses a cubulation with $c$ cubes and has no cubulation with less than $c$ cubes.
\end{defi}

A more flexible definition, strictly related to cubic-complexity is surface-complexity.
A subset $\Sigma$ of $M$ is said to be a {\em Dehn surface} of $M$ if there exists an abstract (possibly non-connected) closed surface $S$ and a transverse immersion $f\co S\to M$ such that $\Sigma = f(S)$.
By transversality, the number of pre-images of a point of $\Sigma$ under such an $f$ is 1, 2 or 3; so there are three types of points in $\Sigma$, depending on this number (in Fig.~\ref{fig:neigh_Dehn_surf} the regular neighbourhoods, in $M$, of the points are shown); they are called {\em simple}, {\em double} or {\em triple}, respectively.
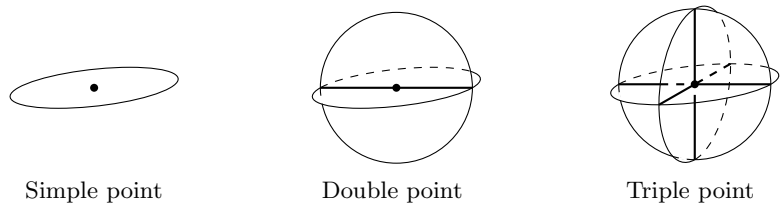
\begin{figure}[t]
\centerline{
\begin{tabular}{ccc}
\begin{minipage}[c]{3.5cm}{\small{\begin{center}
\begin{tikzpicture}[coordinate 3d 30gradi, scale=0.5]
\begin{scope}[canvas is xy plane at z=0]
\draw[lato] (0,0) circle [radius=1];
\end{scope}
\draw (0,0,0) pic {punto};
\end{tikzpicture}
\end{center}}} \end{minipage}
&
\begin{minipage}[c]{3.5cm}{\small{\begin{center}
\begin{tikzpicture}[coordinate 3d 30gradi, scale=0.5]
\begin{scope}[canvas is yz plane at x=0]
\draw[lato] (-1,0) arc [start angle=180, end angle=192.6-360, x radius=1, y radius=1];
\draw[lato tratteggiato] (-1,0) arc [start angle=180, end angle=192.6, x radius=1, y radius=1];
\end{scope}
\begin{scope}[canvas is xy plane at z=0]
\draw[lato] (0,-1) arc [start angle=-90, end angle=143, x radius=1, y radius=1];
\draw[lato tratteggiato] (0,-1) arc [start angle=270, end angle=143, x radius=1, y radius=1];
\end{scope}
\draw[lato evidenziato] (0,-1,0) -- (0,1,0);
\draw (0,0,0) pic {punto};
\end{tikzpicture}
\end{center}}} \end{minipage}
&
\begin{minipage}[c]{3.5cm}{\small{\begin{center}
\begin{tikzpicture}[coordinate 3d 30gradi, scale=0.5]
\begin{scope}[canvas is yz plane at x=0]
\draw[lato] (0,-1) arc [start angle=-90, end angle=180, x radius=1, y radius=1];
\draw[lato] (0,-1) arc [start angle=-90, end angle=-180, x radius=1, y radius=1];
\draw[thick, white] (0,-1) arc [start angle=-90, end angle=-106.8, x radius=1, y radius=1];
\draw[thick, white] (-1,0) arc [start angle=180, end angle=192.6, x radius=1, y radius=1];
\draw[lato tratteggiato] (0,-1) arc [start angle=-90, end angle=-106.8, x radius=1, y radius=1];
\draw[lato tratteggiato] (-1,0) arc [start angle=180, end angle=192.6, x radius=1, y radius=1];
\end{scope}
\begin{scope}[canvas is xy plane at z=0]
\draw[lato] (0,-1) arc [start angle=-90, end angle=143, x radius=1, y radius=1];
\draw[lato tratteggiato] (0,-1) arc [start angle=270, end angle=143, x radius=1, y radius=1];
\end{scope}
\begin{scope}[canvas is xz plane at y=0]
\draw[lato tratteggiato] (0,-1) arc [start angle=270, end angle=127.5, x radius=1, y radius=1];
\draw[lato] (0,-1) arc [start angle=-90, end angle=127.5, x radius=1, y radius=1];
\end{scope}
\draw[lato evidenziato tratteggiato] (-1,0,0) -- (0,0,0);
\draw[lato evidenziato] (0,0,0) -- (1,0,0);
\draw[lato evidenziato] (0,-1,0) -- (0,-0.455,0);
\draw[lato evidenziato tratteggiato] (0,-0.455,0) -- (0,0,0);
\draw[lato evidenziato] (0,0,0) -- (0,1,0);
\draw[lato evidenziato] (0,0,-1) -- (0,0,-0.246);
\draw[lato evidenziato tratteggiato] (0,0,-0.246) -- (0,0,0);
\draw[lato evidenziato] (0,0,0) -- (0,0,1);
\draw (0,0,0) pic {punto};
\end{tikzpicture}
\end{center}}} \end{minipage}
\\
\begin{minipage}[t]{3.5cm}{\small{\begin{center}
Simple point
\end{center}}}\end{minipage} &
\begin{minipage}[t]{3.5cm}{\small{\begin{center}
Double point
\end{center}}}\end{minipage} &
\begin{minipage}[t]{3.5cm}{\small{\begin{center}
Triple point
\end{center}}}\end{minipage}
\end{tabular}}
  \caption{Neighbourhoods of points (marked by thick dots) of a Dehn surface (the lines of double points are drawn thick).}
  \label{fig:neigh_Dehn_surf}
\end{figure}
A Dehn surface $\Sigma$ of $M$ is called {\em quasi-filling} if $M \setminus \Sigma$ is made up of balls.
\begin{defi}
The {\em surface-complexity} of $M$ is equal to $c$ if $M$ possesses a quasi-filling Dehn surface with $c$ triple points and has no quasi-filling Dehn surface with less than $c$ triple points.
\end{defi}

The surface-complexity satisfies some properties, but for the purpose of this paper we recall only the following result, proved in~\cite{Amendola:surf_compl,Amendola:sc1,Korablev-Kazakov:cubic_complexity_two}.
\begin{theorem}\label{theorem:surface_complexity_properties}
The surface-complexity of a (connected and closed) \ptwoirred\ 3-manifold, different from the sphere $S^3$, the projective space $\matRP^3$ and the Lens space $\Lfourone$, is equal to the cubic-complexity of $M$.
The three manifolds $S^3$, $\matRP^3$ and $\Lfourone$ have surface-complexity zero.
There are 11 (connected and closed) \ptwoirred\ orientable 3-manifolds with surface-complexity one, and 80 with surface-complexity two.
\end{theorem}

For the sake of completeness, we mention that the three manifolds $S^3$, $\matRP^3$ and $\Lfourone$ have cubic-complexity one~\cite{Korablev-Kazakov:cubic_complexity_two}.

\section{The classification}

The main result of this paper is the following.
\begin{theorem}\label{theorem:sc1_nonor}
There are 4 (connected and closed) \ptwoirred\ non-orientable 3-manifolds with surface-complexity one (and hence with cubic-complexity one): they are the four flat ones.
\end{theorem}

The definitions for the flat and the Sol geometries (the latter being mentioned below) and the relations with Seifert fibrations can be found in~\cite{Scott}.
We just recall that each of the four flat manifolds has three Seifert fibrations, which can be easily visualised by means of the cubuluation (see Table~\ref{tab:cubul_flat} below).

\paragraph{Matveev complexity and triangulations}
In the proof we will use the {\em Matveev complexity} of $M$, defined in~\cite{Matveev:compl}.
We do not need all details (see~\cite{Matveev:book} for a comprehensive treatise); we only need the fact that for a \ptwoirred\ $M$ distinct from $S^3, \matRP^3, L_{3,1}$ the Matveev complexity equals the minimum number of tetrahedra needed to triangulate $M$.
In~\cite{Amendola-Martelli:small_non_orient,Amendola-Martelli:non_orient_7,Burton:small_non_orient} the list of non-orientable \ptwoirred\ 3-manifolds with Matveev complexity at most $7$ is given.
\begin{remark}\label{rem:compl_6_no_or}
For the aim of this paper, we need only the list up to complexity 6: i.e.
\begin{itemize}
\item
no non-orientable \ptwoirred\ 3-manifold has Matveev complexity less than 6, and
\item
the non-orientable \ptwoirred\ 3-manifolds with Matveev complexity 6 are the four flat ones and the torus bundle with monodromy $\begin{psmallmatrix} 1 & 1 \\ 1 & 0 \end{psmallmatrix}$, which is a Sol manifold.
\end{itemize}
\end{remark}

\paragraph{From cubulations to triangulations}
Triangulations and cubulations are related to each other.
There are a few ways to obtain a triangulation from a cubulation of a manifold.
A first simple construction is shown in~\cite{Amendola:surf_compl}, while a cheeper one is shown in~\cite{Tarkaev:cubic_compl}.
If we start from a cubulation with $c$ cubes, the former construction leads to a triangulation with a number of tetrahedra between $5c$ and $8c$, while the latter to a triangulation with exactly $6c$ tetrahedra.
We will need a finer analysis of the triangulation obtained if we start from a cubulation with one cube, and we will use the ideas of both constructions.

Let us start from a cubulation of $M$ with one cube.
It can be constructed by starting from the abstract cube and by identifying the boundary squares in pairs.
Consider now a triangulation of the abstract cube such that the induced triangulation of each boundary square is composed of two triangles: we will call such a triangulation a {\em block}.
In each square the triangulation is unambiguously defined by the diagonal that is the common edge of the two triangles.
We will call the set of these diagonals a {\em diagonal pattern}.
When we identify two squares to get $M$, either the diagonal (and hence the two triangles) match or not.
If the three pairs of diagonals match, we get a triangulation of $M$.
Otherwise, we will change the block.

We will use the four blocks that are described in Table~\ref{tab:blocks}.
\begin{table}[t]
  \centerline{
\begin{tabular}{llcc}
\toprule
Name & Description & \parbox[t]{2cm}{\begin{center}Number of tetrahedra\end{center}} & Picture \\[-6pt]
\midrule
\parbox[c][2.5cm]{2.1cm}{\em 5-tetrahedron\\ block} & \parbox[c]{3cm}{Four tetrahedra glued along the faces of a central one} & 5 &
\parbox[c]{4.5cm}{\centerline{\begin{tikzpicture}[coordinate 3d 30gradi, figura cubo]
\draw [lato] (1,0,0) -- (1,1,0) -- (0,1,0) -- (0,1,1) -- (0,0,1) -- (1,0,1) -- cycle;
\draw [lato tratteggiato] (0,0,0) -- (0,1,0) (0,0,0) -- (1,0,0) (0,0,0) -- (0,0,1);
\draw [lato] (1,0,1) -- (1,1,1) (1,1,1) -- (0,1,1) (1,1,0) -- (1,1,1);
\draw [lato] (0,0,1) -- (1,1,1) (0,1,0) -- (1,1,1) (1,0,0) -- (1,1,1);
\draw [lato tratteggiato] (0,0,1) -- (0,1,0) (1,0,0) -- (0,1,0) (1,0,0) -- (0,0,1);
\end{tikzpicture}}}
\\
\parbox[c][2.5cm]{2.1cm}{\em flipped\\ block} & \parbox[c]{3cm}{Obtained from the previous one by gluing a tetrahedron along the two triangles of a square} & 6 &
\parbox[c]{4.5cm}{\centerline{\begin{tikzpicture}[coordinate 3d 30gradi, figura cubo]
\draw [lato] (1,0,0) -- (1,1,0) -- (0,1,0) -- (0,1,1) -- (0,0,1) -- (1,0,1) -- cycle;
\draw [lato tratteggiato] (0,0,0) -- (0,1,0) (0,0,0) -- (1,0,0) (0,0,0) -- (0,0,1);
\draw [lato] (1,0,1) -- (1,1,1) (1,1,1) -- (0,1,1) (1,1,0) -- (1,1,1);
\draw [lato] (0,0,1) -- (1,1,1) (0,1,0) -- (1,1,1) (1,0,0) -- (1,1,1);
\draw [lato tratteggiato] (0,0,1) -- (0,1,0) (1,0,0) -- (0,1,0) (1,0,0) -- (0,0,1);
\begin{scope}[shift={(0,-0.5,0)}]
\draw [lato] (1,0,0) -- (0,0,1) -- (1,-0.8,1) -- (0,-0.8,0) -- cycle;
\draw [lato] (1,0,0) -- (1,-0.8,1);
\draw [lato tratteggiato] (0,0,1) -- (0,-0.8,0);
\end{scope}
\draw [freccia incollamento] (0.33,-0.27-0.5,0.33) -- (0.33,0,0.33);
\draw [freccia incollamento] (0.66,-0.27-0.5,0.66) -- (0.66,0,0.66);
\end{tikzpicture}}}
\\
\parbox[c][3cm]{2.1cm}{\em 5-valent\\ block} &\parbox[c]{3cm}{Obtained from the star of a 5-valent edge by gluing a tetrahedron along one of the triangles of the boundary} & 6 &
\parbox[c]{4.5cm}{\centerline{\begin{tikzpicture}[coordinate 3d 30gradi, figura cubo]
\draw [lato] (1,0,0) -- (1,1,0) -- (0,1,0) -- (1,1,1) -- (0,0,1) -- (1,0,1) -- cycle;
\draw [lato tratteggiato] (0,0,0) -- (0,1,0) (0,0,0) -- (1,0,0) (0,0,0) -- (0,0,1);
\draw [lato] (1,1,0) -- (1,1,1) (1,1,1) -- (1,0,1) (1,0,1) -- (1,1,0);
\draw [lato tratteggiato] (1,0,0) -- (0,1,0) (0,0,0) -- (1,0,1) (0,0,1) -- (0,1,0);
\draw [lato evidenziato tratteggiato] (1,0,1) -- (0,1,0);
\begin{scope}[shift={(0,0.5,0)}]
\draw [lato] (0,0,1) -- (0,1,0) -- (0,1,1) -- cycle;
\draw [lato] (0,0,1) -- (1,1,1) -- (0,1,0) (1,1,1) -- (0,1,1);
\end{scope}
\draw [freccia incollamento] (0.33,0.66+0.5,0.66) -- (0.33,0.66,0.66);
\end{tikzpicture}}}
\\
\parbox[c][3.5cm]{2.1cm}{\em 4-valent\\ block} &\parbox[c]{3cm}{Obtained from the star of a $4$-valent edge (which is an octahedron) by gluing two tetrahedra along two opposite triangles of the boundary} & 6 &
\parbox[c]{4.5cm}{\centerline{\begin{tikzpicture}[coordinate 3d 30gradi, figura cubo]
\draw [lato] (0,0,0) -- (1,1,0) -- (0,1,0) -- (1,1,1) -- (0,0,1) -- (1,0,1) -- cycle;
\draw [lato tratteggiato] (0,0,0) -- (0,0,1) (0,0,0) -- (0,1,0);
\draw [lato] (1,1,0) -- (1,1,1) (1,1,1) -- (1,0,1);
\draw [lato] (1,0,1) -- (1,1,0);
\draw [lato tratteggiato] (0,0,1) -- (0,1,0);
\draw [lato evidenziato tratteggiato] (1,0,1) -- (0,1,0);
\begin{scope}[shift={(0,0.5,0)}]
\draw [lato] (0,0,1) -- (0,1,0) -- (0,1,1) -- cycle;
\draw [lato] (0,0,1) -- (1,1,1) -- (0,1,0) (1,1,1) -- (0,1,1);
\end{scope}
\draw [freccia incollamento] (0.33,0.66+0.5,0.66) -- (0.33,0.66,0.66);
\begin{scope}[shift={(0,-0.5,-0)}]
\draw [lato] (1,0,1) -- (1,1,0) -- (1,0,0) -- cycle;
\draw [lato tratteggiato] (0,0,0) -- (1,0,1) (0,0,0) -- (1,1,0) (0,0,0) -- (1,0,0);
\end{scope}
\draw [freccia incollamento] (0.66,0.33-0.5,0.33) -- (0.66,0.33,0.33);
\end{tikzpicture}}}
\\
\bottomrule
\end{tabular}}
  \caption{Four blocks.}
  \label{tab:blocks}
\end{table}
Note that the name chosen for the flipped block underlines that a diagonal is flipped with respect to the 5-tetrahedron block; however, like the 4-valent block, the flipped block also has a 4-valent internal edge (``the flipped diagonal''), the star of the 4-valent internal edge is an octrahedron, and the block is obtained by gluing two tetrahedra to the octahedron (but not along two opposite triangles).
Note also that the number of tetrahedra of the four blocks is at most 6.

\begin{lemma}\label{lemma:four_blocks}
Each 3-manifold with a cubulation with one cube has a triangulation obtained by gluing the squares of one of the four blocks described above.
\end{lemma}

\begin{proof}
Consider a cubulation of a 3-manifold $M$ with one cube.
Identify the cube with the $5$-tetrahedron block and glue the squares to get $M$.
The gluing pairs the six boundary squares into three pairs (in each of which the two squares are identified to each other).
If we consider the diagonal pattern given by the $5$-tetrahedron block (Fig.~\ref{fig:5tetra_pattern}), the pairing is inherited by the diagonal pattern.
\begin{figure}[t]
  \centerline{
\begin{tikzpicture}[coordinate 3d 30gradi, figura cubo]
\draw [lato] (1,0,0) -- (1,1,0) -- (0,1,0) -- (0,1,1) -- (0,0,1) -- (1,0,1) -- cycle;
\draw [lato tratteggiato] (0,0,0) -- (0,1,0) (0,0,0) -- (1,0,0) (0,0,0) -- (0,0,1);
\draw [lato] (1,0,1) -- (1,1,1) (1,1,1) -- (0,1,1) (1,1,0) -- (1,1,1);
\draw [diagonale] (0,0,1) -- (1,1,1) (0,1,0) -- (1,1,1) (1,0,0) -- (1,1,1);
\draw [diagonale tratteggiato] (0,0,1) -- (0,1,0) (1,0,0) -- (0,1,0) (1,0,0) -- (0,0,1);
\end{tikzpicture}
  }
  \caption{The diagonal pattern of the $5$-tetrahedron block.}
  \label{fig:5tetra_pattern}
\end{figure}
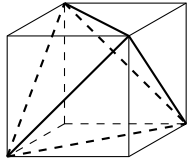
If the three pairs of the diagonals match, we get a triangulation of $M$.
Otherwise, one, two or three of them do not match.

If one pair of diagonals does not match, we consider the flipped block (i.e.~we add a tetrahedron) in order to change one of the two non-matching diagonals, getting a diagonal pattern (shown, up to symmetry, in Fig.~\ref{fig:non-matching_patterns}-a) whose pairs of diagonals match, and hence getting a triangulation of $M$.
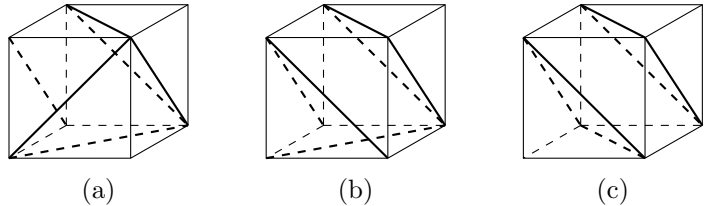
\begin{figure}[t]
  \centerline{
\begin{tabular}{c@{\hspace*{1cm}}c@{\hspace*{1cm}}c}
\begin{tikzpicture}[coordinate 3d 30gradi, figura cubo]
\draw [lato] (1,0,0) -- (1,1,0) -- (0,1,0) -- (0,1,1) -- (0,0,1) -- (1,0,1) -- cycle;
\draw [lato tratteggiato] (0,0,0) -- (0,1,0) (0,0,0) -- (1,0,0) (0,0,0) -- (0,0,1);
\draw [lato] (1,0,1) -- (1,1,1) (1,1,1) -- (0,1,1) (1,1,0) -- (1,1,1);
\draw [diagonale] (0,0,1) -- (1,1,1) (0,1,0) -- (1,1,1) (1,0,0) -- (1,1,1);
\draw [diagonale tratteggiato] (0,0,1) -- (0,1,0) (1,0,0) -- (0,1,0) (1,0,1) -- (0,0,0);
\end{tikzpicture}
&
\begin{tikzpicture}[coordinate 3d 30gradi, figura cubo]
\draw [lato] (1,0,0) -- (1,1,0) -- (0,1,0) -- (0,1,1) -- (0,0,1) -- (1,0,1) -- cycle;
\draw [lato tratteggiato] (0,0,0) -- (0,1,0) (0,0,0) -- (1,0,0) (0,0,0) -- (0,0,1);
\draw [lato] (1,0,1) -- (1,1,1) (1,1,1) -- (0,1,1) (1,1,0) -- (1,1,1);
\draw [diagonale] (0,0,1) -- (1,1,1) (0,1,0) -- (1,1,1) (1,0,1) -- (1,1,0);
\draw [diagonale tratteggiato] (0,0,1) -- (0,1,0) (1,0,0) -- (0,1,0) (1,0,1) -- (0,0,0);
\end{tikzpicture}
&
\begin{tikzpicture}[coordinate 3d 30gradi, figura cubo]
\draw [lato] (1,0,0) -- (1,1,0) -- (0,1,0) -- (0,1,1) -- (0,0,1) -- (1,0,1) -- cycle;
\draw [lato tratteggiato] (0,0,0) -- (0,1,0) (0,0,0) -- (1,0,0) (0,0,0) -- (0,0,1);
\draw [lato] (1,0,1) -- (1,1,1) (1,1,1) -- (0,1,1) (1,1,0) -- (1,1,1);
\draw [diagonale] (0,0,1) -- (1,1,1) (0,1,0) -- (1,1,1) (1,0,1) -- (1,1,0);
\draw [diagonale tratteggiato] (0,0,1) -- (0,1,0) (0,0,0) -- (1,1,0) (1,0,1) -- (0,0,0);
\end{tikzpicture}
\\
(a) & (b) & (c)
\end{tabular}
  }
  \caption{The diagonal pattern if one (a), two (b) or three (c) pairs of diagonals do not match.}
  \label{fig:non-matching_patterns}
\end{figure}

If two pairs of diagonals do not match, it is easy to prove that there are two diagonals, one for each non-matching pair, in adjacent squares.
If we change these two diagonals, we get (up to symmetry) the diagonal pattern shown in Fig.~\ref{fig:non-matching_patterns}-b.
Therefore, if we consider the $5$-valent block, whose pairs of diagonals match, we get a triangulation of $M$.

Finally, if all of the three pairs of diagonals do not match, it is easy to prove that there are three diagonals, one for each non-matching pair, in squares that share a vertex of the cube.
There are two possibilities: either the vertex belong to all of the three diagonals, or it does not belong to any of them.
In the former case we change the three diagonals, in the latter case we change the other three diagonals.
In both cases we get (up to symmetry) the diagonal pattern shown in Fig.~\ref{fig:non-matching_patterns}-c.
Therefore, if we consider the $4$-valent block, whose pairs of diagonals match, we get a triangulation of $M$.

In all cases, we have got a triangulation obtained by gluing the squares of one of the four blocks described above, and the proof is complete.
\end{proof}

\paragraph{Proof of Theorem~\ref{theorem:sc1_nonor}}
We can now prove the main result of the paper.
\begin{proof}
First of all, we note, by means of Theorem~\ref{theorem:surface_complexity_properties}, that the surface-complexity of a non-orientable \ptwoirred\ 3-manifold equals the cubic-complexity, because there is no such manifold with surface-complexity or cubic-complexity zero.

Consider then a non-orientable \ptwoirred\ 3-manifold $M$ with surface-complexity one.
Since $M$ has a cubulation with one cube, by Lemma~\ref{lemma:four_blocks}, we have that the Matveev complexity of $M$ is at most 6, so, by Remark~\ref{rem:compl_6_no_or}, $M$ is either one of the four flat manifolds or the torus bundle with monodromy $\begin{psmallmatrix} 1 & 1 \\ 1 & 0 \end{psmallmatrix}$.
In Table~\ref{tab:cubul_flat} we have shown a cubulation of the four flat manifolds with one cube, so they have cubic-complexity (and surface-complexity) one.
\begin{table}[t]
  \centerline{
\begin{tabular}{llc}
\toprule
Burton notation~\cite{Burton:small_non_orient} &
Regina's notation~\cite{Burton_et_al:regina} &
Cubulation \\
\midrule
$K^2\times S^1$ & 
\parbox[c][2cm]{4.5cm}{\texttt{KB\ x\ S1}\\ \texttt{A=\ x\ S1}\\ \texttt{T\ x\detokenize{~} S1} } &
\parbox[c][2.5cm]{2.5cm}{\centerline{\begin{tikzpicture}[coordinate 3d 30gradi, figura cubo]
\draw [lato] (1,0,0) -- (1,1,0) -- (0,1,0) -- (0,1,1) -- (0,0,1) -- (1,0,1) -- cycle;
\draw [lato tratteggiato] (0,0,0) -- (0,1,0) (0,0,0) -- (1,0,0) (0,0,0) -- (0,0,1);
\draw [lato] (1,0,1) -- (1,1,1) (1,1,1) -- (0,1,1) (1,1,0) -- (1,1,1);
\draw (0,0.5,0.5) node[xscale=-1] {\letteraincollamento{R}};
\draw (1,0.5,0.5) node {\letteraincollamento{R}};
\draw (0.5,0,0.5) node[yslant=0.58] {\letteraincollamento{G}};
\draw (0.5,1,0.5) node[yslant=0.58] {\letteraincollamento{G}};
\draw (0.5,0.5,0) node[xslant=1.73, yscale=0.5] {\letteraincollamento{F}};
\draw (0.5,0.5,1) node[xslant=1.73, yscale=0.5] {\letteraincollamento{F}};
\end{tikzpicture}}}
\\
\midrule
$T^2\times I/_{\begin{psmallmatrix} 0&1\\ 1&0 \end{psmallmatrix}}$ &
\parbox[c][2cm]{4.5cm}{\texttt{SFS\ [KB:\ (1,1)]} \\ \texttt{M\detokenize{_} x\ S1} \\ \texttt{SFS\ [T/o2:\ (1,1)]} } &
\parbox[c][2.5cm]{2.5cm}{\centerline{\begin{tikzpicture}[coordinate 3d 30gradi, figura cubo]
\draw [lato] (1,0,0) -- (1,1,0) -- (0,1,0) -- (0,1,1) -- (0,0,1) -- (1,0,1) -- cycle;
\draw [lato tratteggiato] (0,0,0) -- (0,1,0) (0,0,0) -- (1,0,0) (0,0,0) -- (0,0,1);
\draw [lato] (1,0,1) -- (1,1,1) (1,1,1) -- (0,1,1) (1,1,0) -- (1,1,1);
\draw (0,0.5,0.5) node {\letteraincollamento{R}};
\draw (1,0.5,0.5) node {\letteraincollamento{R}};
\draw (0.5,0,0.5) node[yslant=0.58] {\letteraincollamento{G}};
\draw (0.5,1,0.5) node[yslant=0.58] {\letteraincollamento{G}};
\draw (0.5,0.5,0) node[xslant=1.73, yscale=0.5] {\letteraincollamento{F}};
\draw (0.5,0.5,1) node[xslant=1.73, yscale=0.5, xscale=-1, rotate=-90] {\letteraincollamento{F}};
\end{tikzpicture}}}
\\
\midrule
$K^2\times I/_{\begin{psmallmatrix} 1&0\\ 0&-1 \end{psmallmatrix}}$ &
\parbox[c][2cm]{4.5cm}{\texttt{KB/n3\ x\detokenize{~} S1} \\ \texttt{A=/o2\ x\detokenize{~} S1} \\ \texttt{SFS\ [D\detokenize{_}:\ (2,1)\ (2,1)]} } &
\parbox[c][2.5cm]{2.5cm}{\centerline{\begin{tikzpicture}[coordinate 3d 30gradi, figura cubo]
\draw [lato] (1,0,0) -- (1,1,0) -- (0,1,0) -- (0,1,1) -- (0,0,1) -- (1,0,1) -- cycle;
\draw [lato tratteggiato] (0,0,0) -- (0,1,0) (0,0,0) -- (1,0,0) (0,0,0) -- (0,0,1);
\draw [lato] (1,0,1) -- (1,1,1) (1,1,1) -- (0,1,1) (1,1,0) -- (1,1,1);
\draw (0,0.5,0.5) node[xscale=-1] {\letteraincollamento{R}};
\draw (1,0.5,0.5) node {\letteraincollamento{R}};
\draw (0.5,0,0.5) node[yslant=0.58] {\letteraincollamento{G}};
\draw (0.5,1,0.5) node[yslant=0.58] {\letteraincollamento{G}};
\draw (0.5,0.5,0) node[xslant=1.73, yscale=0.5] {\letteraincollamento{F}};
\draw (0.5,0.5,1) node[xslant=1.73, yscale=0.5, yscale=-1] {\letteraincollamento{F}};
\end{tikzpicture}}}
\\
\midrule
$K^2\times I/_{\begin{psmallmatrix} -1&1\\ 0&-1 \end{psmallmatrix}}$ &
\parbox[c][2cm]{4.5cm}{\texttt{SFS\ [KB/n3:\ (1,1)]} \\ \texttt{M\detokenize{_}/n2\ x\detokenize{~} S1} \\ \texttt{SFS\ [RP2:\ (2,1)\ (2,1)]} } &
\parbox[c][2.5cm]{2.5cm}{\centerline{\begin{tikzpicture}[coordinate 3d 30gradi, figura cubo]
\draw [lato] (1,0,0) -- (1,1,0) -- (0,1,0) -- (0,1,1) -- (0,0,1) -- (1,0,1) -- cycle;
\draw [lato tratteggiato] (0,0,0) -- (0,1,0) (0,0,0) -- (1,0,0) (0,0,0) -- (0,0,1);
\draw [lato] (1,0,1) -- (1,1,1) (1,1,1) -- (0,1,1) (1,1,0) -- (1,1,1);
\draw (0,0.5,0.5) node[rotate=180] {\letteraincollamento{R}};
\draw (1,0.5,0.5) node[rotate=180] {\letteraincollamento{G}};
\draw (0.5,0,0.5) node[yslant=0.58] {\letteraincollamento{G}};
\draw (0.5,1,0.5) node[yslant=0.58] {\letteraincollamento{R}};
\draw (0.5,0.5,0) node[xslant=1.73, yscale=0.5] {\letteraincollamento{F}};
\draw (0.5,0.5,1) node[xslant=1.73, yscale=0.5, xscale=-1, rotate=-90] {\letteraincollamento{F}};
\end{tikzpicture}}}
\\
\bottomrule
\end{tabular}}
  \caption{A cubulation of the four flat manifolds.}
  \label{tab:cubul_flat}
\end{table}

Instead, the torus bundle with monodromy $\begin{psmallmatrix} 1 & 1 \\ 1 & 0 \end{psmallmatrix}$ does not have a cubulation with one cube, so it has cubic-complexity (and hence surface-complexity) greater than one.
In order to prove this (and hence to conclude the proof), we suppose by contradiction that a cubulation with one cube exists.
By Lemma~\ref{lemma:four_blocks} there exists a triangulation obtained by gluing the squares of one of the four blocks shown in Table~\ref{tab:blocks}.
First of all we can rule out the $5$-tetrahedron block because the manifold has complexity $6$ (see Remark~\ref{rem:compl_6_no_or}).
In order to rule out the other three blocks, we will analyse the valences of the edges of the triangulations that can be obtained by means of them.
The valences of the internal edge and of the edges corresponding to the diagonals in the three blocks are listed in Table~\ref{tab:valences}.
\begin{table}[t]
\centerline{\begin{tabular}{cp{3cm}p{3cm}}
\toprule
Block & \centerline{Valence of the}\centerline{internal edge} & \centerline{Valences of the}\centerline{diagonal edges} \\[-10pt]
\midrule
flipped block & \centerline{4} & \centerline{1, 3, 3, 3, 3, 3} \\[-10pt]
$5$-valent block & \centerline{5} & \centerline{2, 2, 2, 2, 3, 3} \\[-10pt]
$4$-valent block & \centerline{4} & \centerline{2, 2, 3, 3, 3, 3} \\[-10pt]
\bottomrule
\end{tabular}}
  \caption{The valences of the internal and diagonal edges of the three blocks.}
  \label{tab:valences}
\end{table}
In each triangulation obtained with these three blocks there is at least one edge with valence $4$: the internal one in the case of the flipped block and in the case of the the $4$-valent block, and the edge corresponding to a diagonal in the case of the $5$-valent block (indeed two of the four diagonals whose corresponding edge has valance 2 must be glued together).
In~\cite{Burton:small_non_orient} the unique triangulation of the torus bundle with monodromy $\begin{psmallmatrix} 1 & 1 \\ 1 & 0 \end{psmallmatrix}$ is shown; for the sake of the clarity, we mention that, as a matter of fact, the matrix used is $\begin{psmallmatrix} -1 & 1 \\ 1 & 0 \end{psmallmatrix}$, but the resulting manifold is the same (as shown in Corollary~A.6 of~\cite{Amendola-Martelli:non_orient_7}).
It has no edge with valence $4$ (see also~\cite{Burton_et_al:regina}), so we have got a contradiction and the theorem is proved.
\end{proof}

\begin{remark}
The proof that the torus bundle with monodromy $\begin{psmallmatrix} 1 & 1 \\ 1 & 0 \end{psmallmatrix}$ has not complexity one can be also given by means of the orientable census of manifolds with cubic-complexity at most two, given in~\cite{Korablev-Kazakov:cubic_complexity_two}.
Indeed a cubulation with one cube of the torus bundle with monodromy $\begin{psmallmatrix} 1 & 1 \\ 1 & 0 \end{psmallmatrix}$ would lift to a cubulation with two cubes of its orientable double covering, which is the torus bundle with monodromy $\begin{psmallmatrix} 2 & 1 \\ 1 & 1 \end{psmallmatrix}$, but this manifold does not appear in the list of~\cite{Korablev-Kazakov:cubic_complexity_two}.
\end{remark}

\begin{small}

\end{small}


\begin{thebibliography}{99}

\bibitem{Amendola:surf_compl}
\textsc{G.~Amendola},
\emph{A 3-Manifold Complexity via Immersed Surfaces},
J. Knot Theory Ramif. \textbf{19}, No.~12 (2010), 1549--1569.

\bibitem{Amendola:sc1}
\textsc{G.~Amendola},
\emph{Orientable closed 3-manifolds with surface-complexity one},
Atti Semin. Mat. Fis. Univ. Modena Reggio Emilia \textbf{57} (2010), 17--26.

\bibitem{Amendola:surf_compl_cpt}
\textsc{G.~Amendola},
\emph{A complexity of compact 3-manifolds via immersed surfaces},
Bollettino dell’Unione Matematica Italiana \textbf{15} (2022), 365--379.

\bibitem{Amendola-Martelli:small_non_orient}
\textsc{G.~Amendola -- B.~Martelli},
\emph{Non-orientable $3$-manifolds of small complexity},
Topology Appl. \textbf{133} (2003), 157--178.

\bibitem{Amendola-Martelli:non_orient_7}
\textsc{G.~Amendola -- B.~Martelli},
\emph{Non-orientable 3-manifolds of complexity up to 7},
Topology Appl. \textbf{150}, No. 1--3 (2005), 179--195.

\bibitem{Burton:small_non_orient}
\textsc{B.~A.~Burton},
\emph{Structures of small closed non-orientable 3-manifold triangulations},
Journal of Knot Theory and Its Ramifications \textbf{16}, No. 05 (2007), 545--574.

\bibitem{Burton:hyperbolic_census_9}
\textsc{B.~A.~Burton},
\emph{The cusped hyperbolic census is complete},
ArXiv, \url{abs/1405.2695}, 2014. URL: \url{https://api.semanticscholar.org/CorpusID:13640292} .

\bibitem{Burton_et_al:regina}
\textsc{B.~A.~Burton -- R.~Budney -- W.~Pettersson et al.},
Regina: Software for low-dimensional topology,
\url{http://regina-normal.github.io/}, 1999--2023.

\bibitem{CaHiWe}
\textsc{P.~J.~Callahan -- M.~V.~Hildebrand -- J.~R.~Weeks},
\emph{A census of cusped hyperbolic $3$-manifolds},
Mathematics of Computation \textbf{68} (1999), 321-332.

\bibitem{Korablev-Kazakov:cubic_complexity_two}
\textsc{F.~G.~Korablev -- A.~A.~Kazakov},
\emph{Manifolds of cubic complexity two},
Sib. \`Elektron. Mat. Izv. \textbf{13} (2016), 1--15.

\bibitem{Lozano-Vigara:subadditivity}
\textsc{A.~Lozano -- R.~Vigara},
\emph{On the subadditivity of Montesinos complexity of closed orientable 3-manifolds},
Rev. R. Acad. Cienc. Exactas Fís. Nat. Ser. A Mat. RACSAM \textbf{109}, No. 2 (2015), 267--279.

\bibitem{Lozano-Vigara:representing}
\textsc{A.~Lozano -- R.~Vigara},
\emph{Representing 3-manifolds by filling Dehn surfaces},
Series on Knots and Everything, 58. World Scientific Publishing Co. Pte. Ltd., Hackensack, NJ, (2016).

\bibitem{Lozano-Vigara:triple_point_spectrum}
\textsc{A.~Lozano -- R.~Vigara},
\emph{The Triple-Point Spectrum of Closed Orientable 3-Manifolds},
Mediterranean Journal of Mathematics \textbf{16}, No. 71 (2019).

\bibitem{Martelli-Petronio:decomposition}
\textsc{B.~Martelli -- C.~Petronio},
\emph{A new decomposition theorem for 3-manifolds},
Ill. J. Math. \textbf{46}, No.~3 (2002), 755--780.

\bibitem{Matveev:compl}
\textsc{S.~V.~Matveev},
\emph{The theory of the complexity of three-dimensional manifolds},
Akad. Nauk Ukrain. SSR Inst. Mat. Preprint, No.~13 (1988).

\bibitem{Matveev:book}
\textsc{S.~V.~Matveev},
Algorithmic topology and classification of 3-manifolds,
Algorithms and Computation in Mathematics, 9. Springer-Verlag, Berlin (2003).

\bibitem{Poincare}
\textsc{H.~Poincaré},
\emph{Analysis situs},
J. de l’Éc. Pol. (2) \textbf{1} (1895), 1-123.

\bibitem{Scott}
\textsc{P.~Scott},
\emph{The geometries of $3$-manifolds},
Bull. London Math. Soc. \textbf{15} (1983), 401--487.

\bibitem{Tarkaev:cubic_compl}
\textsc{V.~V.~Tarkaev},
\emph{On the cubic complexity of three-dimensional polyhedra},
Trudy Inst. Mat. i Mekh. UrO RAN \textbf{17}, No.~1 (2011), 245--250.

\bibitem{Thistlethwaite}
\textsc{M.~Thistlethwaite},
Cusped hyperbolic manifolds with 8 tetrahedra,
\url{http://www.math. utk.edu/~morwen/8tet/}, October 2010.

\bibitem{Vigara:calculus}
\textsc{R.~Vigara},
\emph{A set of moves for Johansson representation of 3-manifolds},
Fund. Math. \textbf{190} (2006), 245--288.

\end{thebibliography}
\end{document}